\documentclass[12pt, reqno]{amsart}
\usepackage{amsmath, amsthm, amscd, amsfonts, amssymb,mathtools}
\usepackage[bookmarksnumbered, colorlinks, plainpages, ]{hyperref}
\usepackage{here}
\usepackage{graphicx,float}
\usepackage{psfrag}
\usepackage{mathrsfs}
\usepackage[T1]{fontenc}
\usepackage[utf8]{inputenc}
\usepackage{hyperref}

\usepackage[all]{xy}
\usepackage{pgfplots}
\usepgfplotslibrary{dateplot}

\usepackage{xspace}

\usepackage{color}
\definecolor{darkgreen}{rgb}{0,0.51,0.11}
\usepackage{blkarray}
\newtheorem{theorem}{Theorem}[section]
\newtheorem{lemma}[theorem]{Lemma}
\newtheorem{proposition}[theorem]{Proposition}

\theoremstyle{definition}
\newtheorem{definition}[theorem]{Definition}
\newtheorem{example}[theorem]{Example}

\theoremstyle{remark}
\newtheorem{remark}[theorem]{Remark}
\numberwithin{equation}{section}
\usepackage{graphicx}
\usepackage{epstopdf}
\usepackage{tikz-cd} 
\begin{document}

%\catchline{}{}{}{}{} % Publisher's Area please ignore

\markboth{Pouya Mehdipour}{A $N$-to-1 Smale Horseshoe}

\title{An $N$-to-1 Smale Horseshoe}
%\author{DAWOUD AHMADI DASTJERDI}
%\address{Faculty of Mathematical Sciences\\ University of Guilan, Iran-Guilan\\}
\author{SANAZ LAMEI}
\address{Faculty of Mathematical Sciences\\ University of Guilan, Iran-Guilan \\lamei@guilan.ac.ir}
\author{POUYA MEHDIPOUR}
\address{Departamento de Matemática\\ Universidade Federal de Viçosa, Brasil-MG\\Pouya@ufv.br} 

%\email{lamei@guilan.ac.ir}%\urladdr{http://www.icmc.sc.usp.br/$\sim$tahzibi}
\maketitle

%\begin{history}
%\received{(to be inserted by publisher)}
%\end{history}

\begin{abstract}
	In this work we extend the Conley-Moser Theorem for $N$-to-1 local diffeomorphisms. By the aim of some extended symbolic dynamics we encode generalized $N$-to-1 horseshoe maps and as a corollary their structural stability is verified.
\end{abstract}
\textbf{keywords:}\textit{ Smale horseshoe, non-invertible dynamics, zip shift maps, structural stability}


\section{Introduction}
%\end{center}

In 1970, Stephen Smale constructed a brilliant example that has become fundamental in the study of dynamical systems \cite{Smale:1976}. 
The so called \textit{Smale Horseshoe}, which characterizes a class of hyperbolic chaotic diffeomorphisms, is known as the hallmark of deterministic chaos. 
In 1969, M. Shub performed a comprehensive study on endomorphisms of compact differentiable manifolds and presented, mainly the structural stability of 
expanding maps \cite{Shub:1969}. In this direction, one can find other works including \cite{Mane and Pugh:1975}, %\cite{Przytycki:1976}, 
\cite{Prezytycki:1977},\cite{Quandt:1988},\cite{Ikegami I:1989}, \cite{Ikegami II:1990}.

In this article, we present the construction and coding of the $N$-to-1 Smale horseshoe map, as an extended version of the well-known  $1$-to-1 Smale horseshoe.
The method we use here to encode the $N$-to-1 horseshoe map is based on an extended version of the bilateral shift over finite alphabets. This extension of shift map 
  is called "zip shift" and is a local homeomorphism. More precisely, we show 
 the topological conjugation of such horseshoe map with an $N$-to-1 zip shift map. The presence of a strong transversality regardless of the choice of the orbits, 
 together with the density of hyperbolic periodic  points,  which is observable in the construction, seems pleasant. It is well known that these two conditions are equivalent to  the structural stability for 
 diffeomorphisms in the $C^1$-topology \cite{Robinson:1976}. It is noteworthy that the zip shift method is promising to perform the $N$-to-1 structural stability 
 of horseshoe map, compatible with its 1-to-1 version for which, the structural stability is well known.  (Section \ref{sec:5}).

The relevance of this construction is based on the natural and intrinsic topological conjugacy with some zip shift map. 
Substantially, without  a zip shift map,  an $N$-to-1 Smale horseshoe map (Section \ref{Sec:2}), is semi-conjugate to a one-sided shift over $2N$ symbols or 
with its inverse limit map. However, a topological conjugacy reveals its real dynamics and highlights some unknown properties.  

In \cite{Lamei and Mehdipour:2021II}, the authors studied the zip shift space from the point of view of symbolic dynamics and coding theory. 
In this same work,  the orbit structure of $N$-to-1 horseshoe maps such as periodic,  homoclinic and heteroclinic orbits are studied. In particular, this information led us to 
several relevant differences that are not presented through the inverse limit studies \cite{Przytycki:1976, Berger and Rovela:2013}. 
This essential difference between the inverse limit space data and the zip shift space data requires careful studies on both spaces. 
We hope that the zip shift method \cite{Lamei and Mehdipour:2021II} can shed a new light on the studies related to the endomorphisms, 
from the point of view of the Smale hyperbolic theory as well as, some of their classification problems.  

In what follows, in Section \ref{Sec:2}, we illustrate the example of an $N$-to-1 horseshoe. In Section \ref{sec:3}, we bring the definition of the local homeomorphism zip shift map 
and the zip shift space defined over two sets of alphabets. Furthermore, we show that zip shift maps exhibit Devaney's chaos. 
In Section \ref{sec:4}, we adapt the Conley-Moser condition for a two-dimensional $N$-to-1 map. The main theorem of Section \ref{sec:4} provides sufficient conditions for the
 existence of an invariant Cantor set for an $N$-to-1 local diffeomorphism. The main references of Section \ref{sec:4} are \cite{Wiggins:1990} and \cite{Moser and Holmes:1973}.
  We enclose the paper verifying the structural stability properties of the $N$-to-1 horseshoe map in Section \ref{sec:5}.

\section{An $N$-to-1 Smale Horseshoe}\label{Sec:2}
In this section we extend the construction of the $1$-to-1 Smale horseshoe to an $N$-to-1 version. 
The interested reader can find more details about the orbit structure of an $N$-to-1 horseshoe in \cite{Lamei and Mehdipour:2021II}.
\begin{definition}[\textbf{An $N$-to-1 local homeomorphism}]\label{Def:p-t-1}
	Let $X$ be a compact metric space and $\{X_i\}_{i=1}^{N}$ be a disjoint collection  of connected subsets of $X$, where the union is contained in or possibly equals  $X$. 
Then $f:X\to X$ is said to be an $N$-to-1 local homeomorphism, when exists local dynamics $f_i: X_i\to  f(X_i),$  which is a homeomorphism for all $i=1,\cdots, N$. 
If the maps $f_i$ are $C^r$-diffeomorphisms, then $f$ is called an  $N$-to-1 $C^r$ local diffeomorphism. 
\end{definition}

\begin{example}\label{ex:prin}
	
Let $D\subset \mathbb{R}^2$ be a closed disk and $Q\subset D$ be the unit square with $f:D\to D$ an $N$-to-1 local homeomorphism. 
Assume that there exist $N$ rectangles $Q_{i}, i=0,\cdots, N-1$ such that $Q=\cup_{i=1}^{N}\, Q_i$ and $f_{|_{{Q_{i}}}}:Q_{i}\to f(Q_i)$ is a diffeomorphism.

Here, we describe the construction of an $N$-to-1 Smale Horseshoe. Let $A: \mathbb{R}^2\to \mathbb{R}^2$ be the linear map,
\[A=\left[
\begin{matrix}
\alpha & 0 \\
0 &  \beta
\end{matrix}
\right]\]
where $\alpha= 2N+\epsilon$ for an arbitrary small $\epsilon >0$ and $\beta=1/{\alpha}$. For simplicity, take $N=2$. The region $A(Q)$ is a rectangle of size $\alpha \times \beta$ 
as in Figure \ref{fig:1}.
Let $x=M$ be a line passing through the middle of $A(Q)$. Fold $A(Q)$ from line $x=M$ or take modulus by $\alpha/2$. Denote this action by $F$. Now, 
$F\circ A(Q)$ is a rectangle of size $\alpha/2 \times \beta$.
Bend  this rectangle 
 back  into $Q$ to perform the horseshoe shape  as shown in Figure \ref{fig:1}.
Define the horseshoe map by $f=B\circ F\circ A$ where $B$ stands for bending.
The intersection $f(Q)\cap Q$ gives two horizontal rectangles $H_a$ and $H_b$. The second iteration of $f$ is also illustrated in Figure \ref{fig:1}.

%\vspace{-2cm}

\begin{figure}[h]
	\begin{center}
	\includegraphics[scale=0.7]{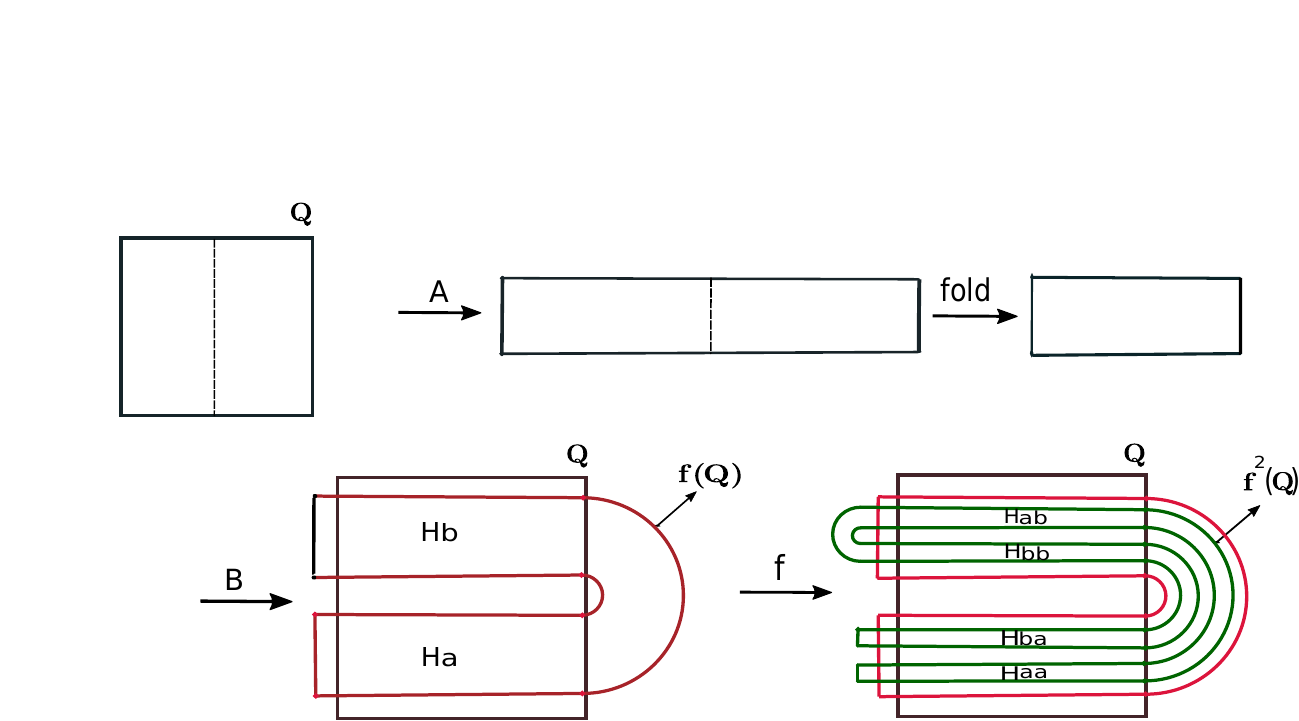}
	\end{center}
	\caption{\it First and second iteration of  a 2-to-1 horseshoe map $f$ on $Q$}
	\label{fig:1}
\end{figure}

Since $f$ is a 2-to-1 map, $f(Q)$ contains two "copies" of the horseshoe shape  and $f^{-1}(Q)$  "separates" those  copies.
It means that $f^{-1}(Q)$ is basically two vertical horseshoe shapes. Label the four vertical rectangles in $f^{-1}(Q)\cap Q$ by 
$V^1$, $V^{1^{'}}$, $V^{2}$ and $V^{2^{'}}$  as in Figure \ref{fig:2}. Note that $f(V^1)=f(V^{1^{'}})=H_a$ and $f(V^2)=f(V^{2^{'}})=H_b$.
The intersection of all forward and backward iterations of $f$ on $Q$ gives a Cantor set $\Lambda$, where any point in $\Lambda$
remains in $\Lambda$ for all  forward and backward iterations of $f$.

Comparing the 1-to-1 Smale horseshoe with the $N$-to-1 horseshoe, there are some differences as well as
 interesting topological and dynamical similarities. For example, they are different in topological entropy or in the number of periodic points of period $ k$. 
 It can be seen that the topological entropy of an $N$-to-1 horseshoe map is equal to $ \log \,2N $ and the number of periodic points with period $ k $ is $ (2N)^k $. 
 On the other hand, 
 both systems are topologically transitive and mixing. Both contain an infinite number of periodic orbits of arbitrary periods and an infinite number of non-periodic orbits (see Sections \ref{sec:3} and \ref{sec:4}).
\begin{figure}[h]
	\begin{center}
	%\hspace{4.3cm}
	\includegraphics[scale=0.5]{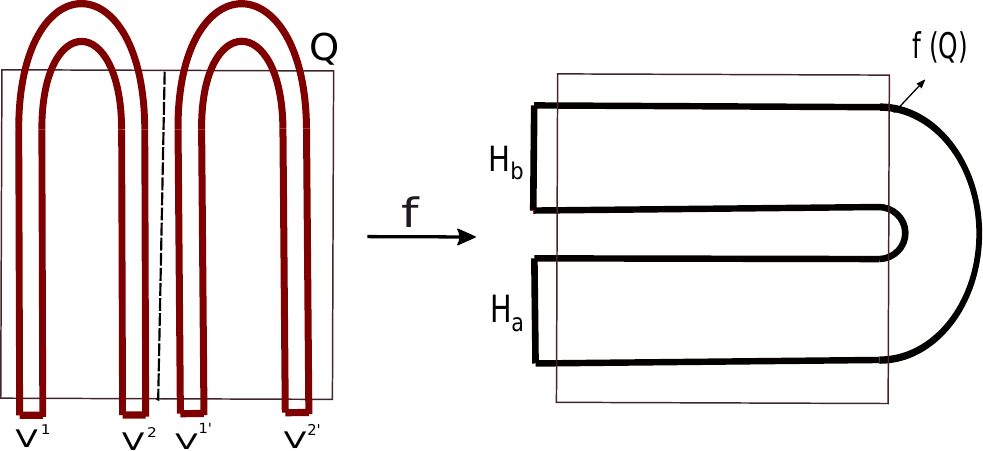}
    \end{center}
	\caption{\it First pre-image of the $2$-to-1 horseshoe map $f$ on $Q$}
	\label{fig:2}
\end{figure}

\end{example}

\section{Full zip shift Space}\label{sec:3}
In this section we describe  an extension of the  two-sided shift map called zip shift map \cite{Lamei and Mehdipour:2021II}. Consider two sets of alphabets $S$ and $Z$ which are 
related by a surjective map  $\tau:S\to Z$. 
%The sets $Z$ and $S$ may contain any finite number of elements. Here, we 
Assume that $\#Z\leq \#S$. 
Define $\Sigma_{S}:=\prod_{i=-\infty}^{+\infty}S_i,$ where $S_i=S$. Consider the two-sided shift $\sigma:\Sigma_{S}\to \Sigma_{S}$. 
Then to any point $ t=(t_i)\in \Sigma_{S}$ correspond a point $ x=(x_i)_{i\in \mathbb{Z}}$, such that
\begin{equation}\label{Z}
	x_i=\begin{cases}
		t_i\in S  &\forall i\geq 0\\
		\tau(t_{i})\in Z &\forall i<0. 
	\end{cases}
\end{equation}
Define 
$\Sigma=\Sigma_{Z,S}:= \{x=(x_i)_{i\in \mathbb{Z}}\ : x_i \textrm{\ satisfies} \,(\ref{Z})\}.$
For $x,y\in \Sigma$, let  $M:\Sigma\times \Sigma\rightarrow\mathbb{N}\cup \{0\}$ be given by $M(x,y)=\min\{|i|\,:\,\,x_i\neq y_i \}. $ Here we use $\{x_{i}\}\neq \{y_{i}\}$ instead of $x_i\neq y_i$ to ensure definition between two sets of pre-images of points as well.
Then  $d(x,y)=\frac{1}{2^{M(x,y)}}$ defines a metric on $\Sigma$ and $(\Sigma, d)$ induce a topology $\mathcal{T}_d$ on $\Sigma.$ 
For $i,n\in \mathbb{Z}$ and $\ell\in \mathbb{N}\cup \{0\}$, one can define the cylinder sets $C_{i}^{\ell}$ as follows 
\begin{equation}
C_{i}^{\ell}=[s_i\,\cdots\, s_{i+\ell}]=\{x= (x_n)\in \Sigma:x_i=s_i,\cdots,x_{i+\ell}=s_{i+\ell}\},
\end{equation}
where $s_j\in S$  for $i\geq 0$, and $s_j\in Z$, for $i<0$, $i\leq j\leq i+\ell$.
The set of all cylinder sets, form a basis for the topological space $(\Sigma, \mathcal{T}_d)$.
%We define the "Zip shift" map  $\sigma_{\tau}: \Sigma\to  \Sigma$. 
It is not difficult to verify that the metric space $(\Sigma,d)$ is compact, totally disconnected and perfect. Indeed it is a Cantor set.   
%One defines the extended bilateral shift map as follows
The following known Lemma is easy to verify \cite{Wiggins:1990}.
\begin{lemma}\label{dis-equal}
	For $ s\,,t \in \Sigma,$
	\begin{itemize}
		\item suppose that $d(s\,, t)<1/(2^{M+1})$. Then $s_i=t_i$ for all $|i|<M.$
		\item suppose that $s_i=t_i$ for $|i|\leq M.$ Then $d(s\,, t)\leq 1/(2^M).$
	\end{itemize}
\end{lemma}
The two-sided full zip shift map is defined as
\begin{definition}[\textbf{Full zip shift map}]\label{zip shift}
	Let $(\Sigma,d)$ and $\tau:S\to Z$ be as above. Then
	\begin{align}\label{sigmatau}
	\sigma_{\tau}: & \Sigma \longrightarrow \Sigma,\\ 
	\nonumber& (x_n)\longmapsto (x_{n+1})=(\cdots x_{-k}\cdots x_{-1}\,\tau(s_0)\,\textbf{.}\,x_1\cdots x_k\cdots),
	\end{align}
	is called the \textit{full zip shift map}.
	%	Where one defines $\tau: S\to Z$ that $\tau(s_k)=a$ when $c_{s_{k}a}=1$ and  $\tau(s_k)=b$ when $c_{s_{k}a}=0$.
\end{definition}

It is obvious that $\sigma_{\tau}(\Sigma)=\Sigma$. Using the fact that the set of all cylinder sets form a basis for the topological space $\Sigma,$ it can be shown that $\sigma_{\tau}$ is a local homeomorphism. Call $(\Sigma, \sigma_{\tau})$  the \textit{full zip shift space}.\\
\textbf{Notation}.  For the rest of paper, we drop the word "full" for simplicity. 
\begin{example}\label{ex:2}
	Consider the 2-to-1 horseshoe map represented in Figures \ref{fig:1} and \ref{fig:2}. 
To correspond a  zip shift space to $\Lambda$, set $S=\{V^{1},V^{1^{'}},V^{2},V^{2^{'}}\}$ and $Z=\{a,b \}$. The horseshoe map $f$ induces
a surjective map $\tau:S\to Z$ with $\tau(1)=\tau(1^{'})=a$ and $\tau(2)=\tau(2^{'})=b$.
Define $$\Sigma=\Sigma_{Z,S}=\{t|\, t=(\cdots t_{-2} t_{-1} \textbf{.} t_0 t_1 \cdots), \, t_i\in S,\forall i\geq 0,\ \textrm{and}\ t_i\in Z,\, \forall i<0\}.$$
For instance take $ t=(\dots a\,b\,a\,b\,b\,.\,1\,0\,1^{'}\,1\,1\,\dots)$. 
Then $$\sigma_{\tau}((\dots\,a\,b\,a\,b\,b\,.\,1\,2\,1^{'}\,1\,2^{'}\,\dots))=(\dots\,a\,b\,a\,b\,b\,a\,.\,2\,1^{'}\,1\,2^{'}\,\dots).$$ 
There exists a homeomorphism $\phi: \Lambda\to \Sigma$ which is a  topological conjugacy between the horseshoe map $f$ and the zip shift map $\sigma_{\tau}$
(see Subsection \ref{subsec:3.1}).
\end{example}
It is worth mentioning that for an $N$-to-1 Smale horseshoe, one can take, $$S=\{1,2,\dots,N, 1',2',\dots,N'\}\,\,\,\,\, \text{and}\,\,\,\,\,\, Z=\{a,b\},$$ 
where $\tau(1)=\dots=\tau(N)=a$ and $\tau (1^{'})=\dots=\tau(N^{'})=b$.
Besides, in case $S=Z$ and $\tau(s_i)=\textrm{Id}$, one obtains the  two-sided shift which is conjugate to a 1-to-1 
Smale horseshoe map (more details in \cite{Lamei and Mehdipour:2021II}).

\begin{example}\label{ex:1}
Let $f(x)=2x \ \textrm{mod}\, 1$.  According to  Figure \ref{fig:4}, $S=\{0,1\}$ and $Z=\{a\}$. The map $\tau$ is a constant map defined as, $\tau(0)=\tau(1)=a$. 
Any element of the zip shift space $\Sigma_{Z,S}$ have the form $s=(\dots aaa\,\textbf{.}\,s_0s_1s_2\dots)$, where $s_i\in S$. 
Then $\sigma_{\tau}( s)=(\dots aaaa\,\textbf{.}\,s_1s_2\dots)$. In general,  any full one-sided shift on $N$ alphabets, is conjugate to a zip shift map, 
where the surjective map $\tau$ is the constant map $\tau(s_i)=a$.
\end{example}
\begin{figure}[h]
	\centering
	%\hspace{4.3cm}
	\includegraphics[scale=0.5]{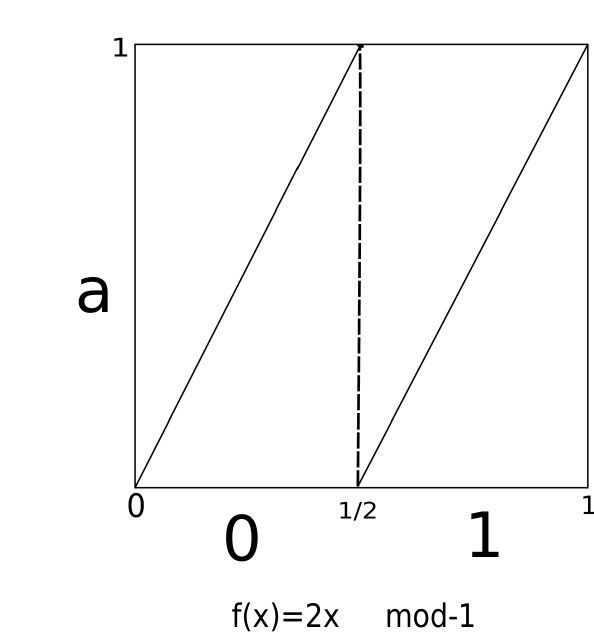}
	%	%\vspace{-1.3cm}
	\caption{\it $Z=\{a\}$ and $S=\{0,1\}$.} 
	\label{fig:4}
\end{figure} 
	
\subsection{Zip Shift Maps are Chaotic}
Topological transitivity, density of periodic points and sensitivity to the initial conditions are properties that characterize the Chaos (Devaney's Definition). 
In what follows we distinguish such properties for zip shift maps. 

\begin{definition}[\textbf{Expansivity}]
	Let $(X,d)$ be a compact metric space and $f:X\to X$ a continuous dynamical system (local homeomorphism). We say that $f$ is an \textit{Expansive} map with expansivity constant $e>0$ if for any $x,y\in X$ exists some $n\in\mathbb{Z}$ such that $d(f^{n}(x),f^{n}(y))>e$. In case of $n<0$ we consider minimum distance of the sets.
\end{definition}

\begin{definition}[\textbf{Periodic and pre-periodic points}]
	Let $(\Sigma, \sigma_{\tau})$ represent a zip shift space defined on two alphabet sets $S,Z$ and with some $\tau:S\to Z$. A  \textit{periodic} point of period $n,$ has the 
form $${p}=(\overline{\tau(p_0)\,\dots\,\tau(p_{{n-1}})}\,\textbf{.}\,\overline{p_0\,p_1\,\dots\,p_{n-1}})$$ where the overline means the repetition. 
For simplicity we may represent it as ${p}=(\overline{p_0\,p_1\,\dots\,p_{n-1}}).$ 

A point $q$ is called a \textit{pre-periodic}  of ${p}$, if there exists some $k>0$ that $f^k(q)=p$.
\end{definition}

\begin{remark}
Note that when $\tau$ is not an injective map,  $\tau^{-1}(\tau (p_{t}))$ can be any choice of the $N$ pre-images of $\tau (p_{t})$ as shown in the following example.
\end{remark}
\begin{example}\label{ex:3}
	In Example \ref{ex:2} the points $p_1=(\overline{ab}\textbf{.}\overline{12})$ and $p_2=(\overline{ab}\textbf{.}\overline{1^{'}2^{'}})$ are periodic points of period 2 and points $q_1=(\overline{ab}\textbf{.}1^{'}2^{'}\overline{12})$ and $q_2=(\overline{ba}\textbf{.}2^{'}1^{'}2\overline{1^{'}2^{'}})$ are pre-periodic points correspondingly associated to $p_1$ and $p_2$. 
\end{example}

\begin{theorem}\label{thm:1}
	The periodic points of a zip shift space $(\Sigma, \sigma_{\tau})$ are dense in $\Sigma.$
\end{theorem}

\begin{proof}
	Let $C_{i}^{\ell}=[s_i\,\dots\, s_{i+\ell}]\subset (\Sigma,d)$.  There are three different possibilities as follows. In each case we find a periodic point $p$ in  $C_{i}^{\ell}$.
	\begin{itemize}
		\item If $i\geq 0$, 
	%Then $p=(\dots,p_{-1}\,\textbf{.}\,p_0,\dots,p_i,\dots,p_{i+\ell},\dots)\in C_{i}^{\ell}.$ 
	take ${p}=(\overline{\tau(p_{0})\,\cdots\,\tau(s_{i+\ell})}\,\textbf{.}\,\overline{p_0\,\dots\,p_i\,\dots\,p_{i+\ell}})=(\overline{p_0\,\dots\,s_i\,\dots\,s_{i+\ell}})$
	which is  a periodic point that belongs to $C_{i}^{\ell}$.
		\item If $i+\ell< 0$,
	% Then $p=(\dots, p_{i},\dots,p_{i+\ell},\dots,p_{-1}\,\textbf{.}\,p_0,\dots)\in C_{i}^{\ell}.$ 
	take ${p}=(\overline{p_i\,\dots\,p_{i+\ell}\,\dots\,p_{-1}}\,\textbf{.}\,\overline{c_0\,\dots\,c_{-(i+1)}})=(\overline{c_0\,\dots\,c_{-(i+1)}})\in C_{i}^{\ell}$ 
	%is a  periodic point that belongs to $C_{i}^{\ell}$ with $i+\ell<0$. Here 
	where $p_i=s_i$,$\dots$,$p_{i+\ell}=s_{i+\ell}$ and $c_0\in \tau^{-1}(p_i)$, $\dots$, $c_{-(i+1)}\in \tau^{-1}(p_{-1})$. 
	Since $\tau$ is an onto map, one can find more than one periodic point depending on the choices of $c_j\in\tau^{-1}(p_{j+i})$ for $0\leq j\leq -(i+1)$. 
		\item If $i<0$ and $i+\ell>0$, then take
	%$p=(\dots, p_{i},\dots,p_{-1}\,\textbf{.}\,p_0,\dots,p_{i+\ell},\dots)\in C_{i}^{\ell}.$ Note that,
		$$p=(\overline{\tau(p_0)\dots\tau(p_{i+\ell})\, p_{i}\,\dots\,p_{-1}}\,\textbf{.}\,\overline{p_0\,\dots\,p_{i+\ell}\,c_{i+\ell +1}\,\dots\,c_{\ell+1 }}),$$ 
	simplified to ${p}=(\overline{p_0\,\dots\,p_{i+\ell}\,c_{i+\ell +1}\,\dots\,c_{\ell +1}})$
		belongs to $C_{i}^{\ell}$. Here $p_i=s_i$,$\dots$, $p_{i+\ell}=s_{i+\ell}$ and $c_{i+\ell +1}\in \tau^{-1}(p_i)$, $\dots$, $c_{\ell +1}\in \tau^{-1}(p_{-1})$. 
	%In the same way, different choices of periodic points are accessible.
	\end{itemize}
	Therefore, the periodic points are dense in the zip shift space $(\Sigma, \sigma_{\tau})$.
\end{proof}
\begin{remark}
Denote the sets of periodic points and pre-periodic points of $\sigma_{\tau}$ by  $\textrm{Per}(\sigma_{\tau})$ and  $\textrm{PPer}(\sigma_{\tau})$respectively. Then $\textrm{Per}(\sigma_{\tau})\subseteq \textrm{PPer}(\sigma_{\tau})$. Therefore, the set of pre-periodic points of a zip shift map is also dense in $\Sigma$.
\end{remark}
	As it is known, the expansiveness implies the sensibility to initial conditions. In what follows we show that zip shift maps are expansive maps. %In \cite{Mehdipour:2022} the author shows that in fact any zip shift of finite type is expansive and has Pseudo Orbit Tracing Property (P.O.T.P.). \textbf{ Moreover, it is shown that zip shifts of finite type are topologically stable.}
\begin{proposition}\label{prop:1}
	The zip shift maps are expansive local homeomorphisms.
\end{proposition}
\begin{proof}
	Let $(\Sigma,\sigma_{\tau})$ represent a zip shift map \eqref{zip shift}. Then  $\sigma_{\tau}$ is expansive with expansivity constant $e=1/2$. In fact for any $x\neq y\in \Sigma$ there exists some $i\in\mathbb{Z}$ such that $d(\{\sigma_{\tau}^i(x)\},\{\sigma_{\tau}^i(y)\})=1>1/2$.  
\end{proof}

%A complete version of following Proposition is accessible in \\ \cite{Mehdipour and Lamei and Vargas:2024 }.
\begin{proposition}\label{prop:2}
Let $X$ be a compact metric space and $f:X\to X$ a local homeomorphism that is topologically conjugated with a zip shift map $\sigma_{\tau}:\Sigma\to \Sigma$. If $\phi: X\to \Sigma$ denotes the conjugacy map, then $f=\phi^{-1}\circ \sigma_{\tau}\circ \phi:X\to X$ is expansive.
\end{proposition}
\begin{proof}
If $e>0$ be the expansivity constant and $x,y \in X$, we need to show that if $d_X(f^n(x),f^n(y))<e$ for all $n\in \mathbb{Z}$, then $x=y$. By uniform continuity of $\phi$, exists some $\delta>0$ such that if $d_X(f^n(x),f^n(y))<\delta$, then $d_\Sigma(\phi\circ f^n(x),\phi\circ f^n(y))<1/2$. Let $e=\delta$ then as $d_\Sigma(\phi\circ f^n(x),\phi\circ f^n(y))=d_\Sigma(\sigma_{\tau}^n\circ \phi(x),\sigma_{\tau}^n\circ \phi(y))$
% $d_\Sigma(\sigma_{\tau}^n\circ \phi (x),\sigma_{\tau}^n\circ \phi(y))<\delta$ implies $d_X(f^n(x),f^n(y))<\epsilon$ for all $n\in\mathbb{Z}$. 
	%(for any $\tilde{x}\in \pi^{-1}(x),\tilde{y}\in \pi^{-1}(y)$, $n\in\mathbb{Z}$).  
%	We have 
%$$d_\Sigma(\phi\circ f^n(x),\phi\circ f^n(y))=d_\Sigma(\sigma_{\tau}^n\circ \phi(x),\sigma_{\tau}^n\circ \phi(y))<\delta$$
	for all $n\in\mathbb{Z}$, one has $\phi(x)=\phi(y)$, which implies $x=y$.
\end{proof}

\begin{definition}[\textbf{Topological Transitivity}]
	Let $X$ be a compact metric space and $f:X\to X$ be a continuous map. Then $f$ is topologically transitive, if for any two non-empty disjoint open subsets $U,V\subset X$,
 there exists some natural number $m$, such that $f^{m}(U)\cap V\neq \emptyset.$ If $X$ has no isolated points then the existence of a forward dense orbit 
 (transitivity) implies the topological transitivity \cite{Akin and Carlson:2012}.
\end{definition}

\begin{definition}[\textbf{Pre-Transitivity}]
	Let $X$ be a compact metric space and $f:X\to X$ be a non-invertible map. Then $f$ is called \textit{pre-transitive}, 
if there exists some $x\in X$, for which, the set of all pre-images of $x$ i.e. $\cup_{n=0}^{+\infty} f^{-n}(x)$, is dense in $X$. 
%The map $f$ is called "Pre-Transitive" if it is pre-transitive with respect to all $ p\in X$. 
\end{definition}

The pre-transitivity is used to study the uniqueness of SRB measures for endomorphisms \cite{Mehdipour: 2018}. 

\begin{theorem}\label{thm:2}
	Any  zip shift map is transitive and pre-transitive.
\end{theorem}

\begin{proof}
	Without any loss of generality, let us assume that $Z=\{a_1,\dots,a_m\}$ and $S=\{0,\dots,\ell-1\}$ where $m\leq l$ and $\tau:S\to Z$ is the associated surjective map. 
%To show that \textbf{$\sigma_{\tau}$ is transitive}, 
Let $\beta_n(S)$ be the set of all blocks in $\Sigma$   of length $n$ with letters in $S$ (in a lexicographical order). Therefore $\#(\beta_n(S))=\ell^n$. Note that, when $u=x_i\ldots x_j$ is a finite block of $x\in \Sigma$ with $i\geq 0$, then 
$\tau(u)$ is defined as $\tau (x_i)\ldots \tau (x_j)$.

	For $n\geq 1$, sort the blocks in $\beta_n(S)$   successively and  consider the following point in $\Sigma_{Z,S}$:
	$$x=(\dots\,\textbf{.}\,x_1^1\,x_1^2\,\dots\,x_1^{\ell}\,\dots\, x_{n}^1\,\dots\,x_{n}^{\ell_{n}}\, \dots),$$ where  $x_{n}^j\in \beta_n(S)$ for $1\leq j\leq \ell^n$. 
	%and $\tau(x_{n}^j)\in Z$. 
 In this way $x$ represents a dense orbit in $\Sigma$.  
	For any cylinder set $C_i^{\ell}$ one can find some $k>0$ that $\sigma_{\tau}^k(x)\in C_{i}^{\ell}$. 
In fact, it is not difficult to verify that if $i\geq 0,$ $\text{\underline{or}},$ if $i+\ell<0$, then there exists some block
 $x_{\ell}=c_i \dots c_{i+\ell}\in \beta_{\ell}(S)$ that coincides with $s_i \dots s_{i+\ell}$ $\text{\underline{or}},$ 
 with $\tau(s_i)\dots \tau(s_{i+\ell})$, such that, for some $k>0$,  $\sigma_{\tau}^{k}(x)\in C_{i}^{\ell}$. 
 Moreover, if $i<0$ and $i+\ell>0$, then again there exists some block $x_{\ell}=c_k \dots c_{k+\ell}\in \beta_{\ell}(S)$ 
 such that $c_k = s_i,\dots, c_{k+i-1}=s_{-1} ,c_{k+i}=s_0,\dots,c_{k+l}=s_{i+l}$. 
 Thereupon, there exists some $k>0$ that $\sigma_{\tau}^{k}(x)\in C_{i}^{\ell}$. Thus, $\sigma$ is transitive.
	
Now we show that $\sigma_{\tau}$ is pre-transitive. As $\sigma_{\tau}$ is a  zip shift map and in general represents a finite-to-1 map, there are infinitely many $x$ with a forward transitive orbit. Consider the following point in $\Sigma_{Z,S}$,
$$x=(\dots\,x_{-n}^{\ell_{n}}\,\dots\, x_{-n}^1\,\dots\,x_{-1}^{\ell}\,\dots\,x_{-1}^1\,\textbf{.}\,x_1^1\,x_1^2\,\dots\,x_1^{\ell}\,\dots\, x_{n}^1\,\dots\,x_{n}^{\ell_{n}}\, \dots),$$ where  $x_{n}^j\in \beta_n(S)$ for $1\leq j\leq \ell^{n}$ and $x_{-n}^j\in \beta_n(Z)$ for $1\leq j\leq m^n$. 

Now let $C_{i}^{\ell}=[s_i\,\dots\, s_{i+\ell}]\subset (\Sigma,d)$ be an arbitrary cylinder set. Again there are three different possibilities that in each case, we find some pre-image $q$ of $x$ that belongs to $C_{i}^{\ell}$. 
\begin{itemize}
	\item If $i\geq 0$, there exists some block $x_{-\ell}=c_i  \dots  c_{i+\ell}\in \beta_{\ell}(Z)$ that $\tau(x_{-\ell})$ coincides with block $s_i \dots s_{i+\ell}$. Indeed, one can find $k>0,$ for which, there exists some $y\in \sigma_{\tau}^{-k}(x)$ that belongs to $C_{i}^{\ell}$. Note that $y$ contains block $s_i \dots s_{i+\ell}$ starting at $i$.
	%Then $p=(\dots,p_{-1}\,\textbf{.}\,p_0,\dots,p_i,\dots,p_{i+\ell},\dots)\in C_{i}^{\ell}.$ 
%	take ${p}=(\overline{\tau(p_{0})\,\cdots\,\tau(s_{i+\ell})}\,\textbf{.}\,\overline{p_0\,\dots\,s_i\,\dots\,s_{i+\ell}})=(\overline{p_0\,\dots\,s_i\,\dots\,s_{i+\ell}})$
	\item If $i+\ell< 0$,
it is enough to choose some block $x_{-(l+1)}=c_{1}\dots c_{i+\ell+1}$ where $c_{2}=s_i,\dots, c_{i+\ell+1}=s_{i+\ell}$. Then there exists $k>0$ that any $y\in \sigma_{\tau}^{-k}(x)$ belongs to $C_{i}^{\ell}$. 
Note that $y$ contains block $s_i \dots s_{i+\ell}$ starting at $i$.
	\item If $i<0$ and $i+\ell>0$, take some $x_{-\ell}=c_1 \dots c_{l}$ where $c_1=s_i, c_{i}=s_{-1}$ and $c_{i+1}=\tau(s_0),\dots, c_{\ell}=\tau(s_{i+\ell})$. Then there exists $k>0$ and $y\in \sigma_{\tau}^{-k}(x)$ that $y\in C_{i}^{\ell}$. Note that in this case the $y$ is unique and contains block $s_{0} \dots s_{i+\ell}$ starting at $0$-entry.
	
\end{itemize}

Therefore, the zip shift map is pre-transitive.
\end{proof}
%{\color{red} Since $\Sigma$ is a Cantor set, it contains no isolated points and as the zip shift maps are continuous, Theorems \ref{thm:1} and \ref{thm:2}  together with the main result of \cite{Banks et al:1992} show that they are chaotic in the sense of Devaney. }

\section{The Conley-Moser Conditions}\label{sec:4}
In this section we express the sufficient conditions in order to have an invariant Cantor set for an $N$-to-1 local homeomorphism, on which the dynamic
 is topologically conjugate to a  zip shift map. 
These conditions were given by Conley and Moser  \cite{Wiggins:1990, Moser and Holmes:1973}, for an invertible map and we aim to extend 
them for $N$-to-1 local homeomorphisms.
\begin{definition}[\textit{Horizontal/Vertical curves}]\label{def:v-h curve}
For $L \in\{h,\,v\}$, let
$0<\mu_L <1$, be a real number. By a horizontal curve, we mean the graph of a function $y=h(x)$ for which $0\leq h(x) \leq 1$ and 
	$$|h(x_1)-h(x_2)|\leq \mu_h|x_1-x_2|,\,\,\,\,\text{for}\,\,\, x_1,x_2\in [0,1].$$
	Similarly, by a vertical curve, we mean the graph of a function $x=v(y)$, for which $0\leq v(y) \leq 1$ and 
	$$|v(y_1)-v(y_2)|\leq \mu_v|y_1-y_2|,\,\,\,\,\text{for}\,\,\,y_1,y_2\in [0,1].$$
\end{definition}
Note that, by Definition \eqref{Def:p-t-1}, an $N$-to-1 local homeomorphism contains $ N $ distinct vertical curves that are mapped homeomorphically onto a single horizontal curve. 
In such cases, we may refer to them as $N$-to-1 horizontal-vertical  curve (\textit{HV-curve}). See Figure \ref {fig:3}, for a 2-to-1 HV-curve.

One can consider strips instead of these curves and obtain $N$-to-1  horizontal-vertical strip abbreviated as \textit{HV-strip}. 
Next, we indicate the two boundary curves of a horizontal strip (resp. vertical strip) by $h(x)$ and $h'(x)$ (resp. $v(y)$ and $v'(y)$). 
Notice that the "$ '\, $" is adopted as a notational convention and should not be confused with derivation symbol.

\begin{definition}[\textit{Horizontal/Vertical strips}]\label{v-h strips}
	Given a pair of non-intersecting horizontal curves, $h(x)$ and $h^{'}(x)$ such that $0\leq h(x)<h^{'}(x)\leq 1$, define a horizontal strip as,
	\begin{eqnarray}\label{HL}
	H=\{(x,y)\in \mathbb{R}^2|\, y\in[h(x),h^{'}(x)],\,x\in[0,1]\}.
	%H_{2}=\{(x,y)\in \mathbb{R}^2|y\in[h_{2}(x),h_{2}^{'}(x)];x\in[0,1]\}.
	\end{eqnarray}
	Similarly, given two non-intercepting vertical curves $v(y)$ and $v^{'}(y)$ such that $0\leq v(y)<v^{'}(y)\leq 1,$ a vertical strip is defined as,
	\begin{eqnarray}\label{VL}
	V=\{(x,y)\in \mathbb{R}^2|\,x\in[v(y),v^{'}(y)],\,y\in[0,1]\}.
	%H_{2}=\{(x,y)\in \mathbb{R}^2|y\in[h_{2}(x),h_{2}^{'}(x)];x\in[0,1]\}.
	\end{eqnarray}
\end{definition}                                                            
Let $|.|$ be the usual distance in $\mathbb{R}^2.$ Then the width of horizontal and vertical strips are defined as follows.  

\begin{align}\label{dv}
d(V)=\max_{y\in[0,1]}\,\,\,|v(y)-v^{'}(y)|,\\
d(H)= \max_{x\in[0,1]}\,\,\,|h(x)-h^{'}(x)|.
\end{align}

\begin{figure}[h]
	\centering
	\includegraphics[scale=0.5]{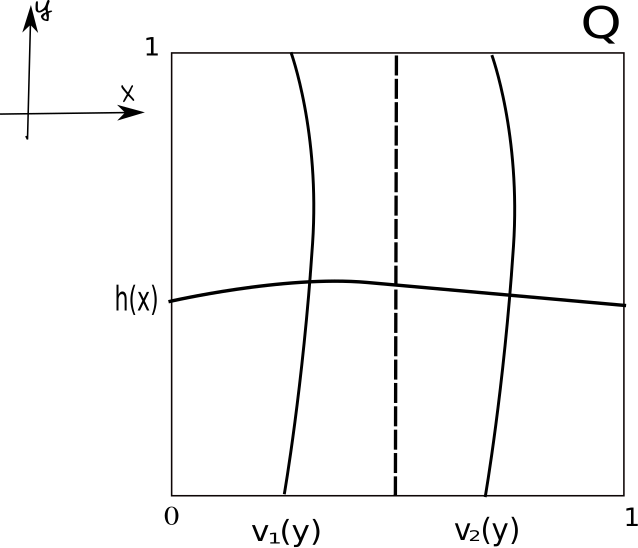}
	%\vspace{-1cm}
	\caption{A 2-to-1 horizontal-vertical curve (HV-curve)}
	\label{fig:3}
\end{figure}
%The following lemmas are important in the construction of the invariant set.
\begin{lemma}\label{v-h curve}
	i) If $H_1 \supset H_2\supset\cdots \supset H_k\supset \cdots$ is a nested sequence of horizontal strips with $d(H_k)\to 0$ as
 $k\to \infty,$ then $H_{\infty}:=\bigcap_{k=1}^{\infty}H_k$ is a horizontal curve.\\
	ii) If $V^{1}_{i} \supset V^{2}_{i}\supset\cdots \supset V^{k}_{i}\supset \cdots$, ($i=1,\cdots,N$) 
	is a nested sequences of vertical strips with $d(V^{k}_{i})\to 0$ as $k\to \infty,$ then $V^{\infty}_i:=\bigcap_{k=1}^{\infty}V^{k}_{i}$ is a vertical curve. 
\end{lemma}
\begin{proof}
	We prove the item (ii). Item (i) can be proved analogously. 
	Let $C_{\mu}[0,1]$ denote the set of Lipschitz functions with Lipschitz constant $\mu$ defined on the interval $[0,1]$. Observe that with the 
	maximum norm as a metric, $C_{\mu}[0,1]$ is a complete metric space.
	Let $x=v_{1}^{k}(y)$ and $x=v_{2}^{k}(y)$ be the boundaries of a vertical strip $V_{i}^k$. 
Consider the sequence $$\{v^{1}_{1}(y),v^{1}_{2}(y),v^{2}_{1}(y),v^{2}_{2}(y),\dots v^{k}_{1}(y),v^{k}_{2}(y),\dots\}.$$
	By Definition \ref{v-h curve} the elements of the above sequence belong to $C_{\mu}[0,1]$. Since $d(V^k)\to 0$ as $k\to\infty,$ it is a Cauchy sequence. 
Therefore, it converges to a unique vertical curve, denoted by $V_{i}^{\infty}.$
\end{proof}
\begin{lemma}\label{v-h intersection}
	For an $N$-to-1 HV-curve, any of the vertical curves intersects the horizontal curve in a unique point.
\end{lemma}
\begin{proof}
	The $N$-to-1 HV-curve can be represented by graphs $$x=v_{1}(y),\,\,x=v_{2}(y),\dots, x=v_{N}(y),y=h(x).$$ 
	Then $h(x)$ intersects $v_1(y),\dots,v_n(y)$ if there exist points $(a_i,b_i)\in Q_i$  such that $a_i=v_{i}(b_i),\,b_i=h(a_i).$ 
Each of the equations $x=v_{i}(h(x)),i=1\dots ,N,$  has a solution on $Q_i, i=1,\dots, N$. We intend to show that these solutions are unique. 
Observe that $I=[0,1]$ is a complete metric space and for each $i,$ the map $v_i\circ h:I\to I$  is a contraction mapping ($0<\mu<1$). 
By applying the Contraction Mapping Theorem \cite{Wiggins:1990}, each of the equations $x=v_{i}(h(x))$ have a unique solution. 
\end{proof}
For horizontal and vertical strips, with boundary curves $h(x),h'(x)$ and $v(y),v'(y)$ respectively, let
\begin{align}\label{strip norms}
\Vert h-h'\Vert:=\max_{x\in[0,1]}|h(x)-h'(x)|,\\
\Vert v-v'\Vert:=\max_{y\in[0,1]}|v(y)-v'(y)|.
\end{align}
\begin{lemma}\label{4.8}
	Let $\mu$ be the Lipschitz constant in Definition \ref{def:v-h curve}. For $L\in\{h,\,v\}$, let $H$ and $V$ be a pair of horizontal and vertical strips with $\mu_L$ Lipschitz 
boundary curves $h(x), h'(x)$ and $v(x), v'(x)$ respectively. Denote the intersection point of curves $h(x)$ and $v(y)$ by  $z_1=(x_1,y_1)$ 
and   the intersection point of curves $h'(x)$ and $v'(y)$ by $z_2=(x_2,y_2)$, where  $x_1,x_2, y_1,y_2\in I$. Then,
	\begin{align}\label{eq:4.8}
	|z_1-z_2|\leq \frac{1}{1-\mu}(\Vert v-v'\Vert+\Vert h-h'\Vert).
	%		|y_1-y_2|\leq \frac{1}{1-\mu}(\Vert h_1-h_1^{'}\Vert+\Vert v_1-v_1^{'}\Vert).
	\end{align}	
\end{lemma}
\begin{proof}
The proof is an adaptation of the proof given in \cite{Moser and Holmes:1973}.
	Observe that,
	\begin{align}\label{eq:1}
	|x_1-x_2|=|v(y_1)-v'(y_2)|
	& \leq |v(y_1)-v(y_2)|+|v(y_2)-v'(y_2)|\nonumber\\
	& \leq\mu_v|y_1-y_2|+\Vert v-v'\Vert
	\end{align}
	Similarly, 
	\begin{align}\label{eq:2}
	|y_1-y_2|=|h(x_1)-h'(x_2)|
	%		& \leq |h_1(x_1)-h_1(x_2)|+|h_1(x_2)-h_2(x_2)|\nonumber\\
	& \leq\mu_h|x_1-x_2|+\Vert h-h'\Vert.
	\end{align}
	 Statements (\ref{eq:1}) and (\ref{eq:2}) together with the fact that $0<\mu=\mu_h\mu_v<1$, gives
%	implies $\frac{\mu}{1-\mu^2}<\frac{1}{1-\mu},$
	\begin{align*}
	|z_1-z_2|\leq \frac{1}{1-\mu}(\Vert v-v'\Vert+\Vert h-h'\Vert).
	%		|y_1-y_2|\leq \frac{1}{1-\mu}[\Vert h_1-h_1^{'}\Vert+\Vert v_1-v_1^{'}\Vert].
	\end{align*}
\end{proof}
Let  $D\subset\mathbb{R}^2$ be a closed disk and $Q\subset D$ be a rectangle (unit rectangle) where $f:D\to D$ is an $N$-to-1 local homeomorphism. Here, for the sake of simplicity, we assume that $Z$ has two elements, but in general, for such an $N$-to-1 map, $Z$ can have any finite number of elements with $ \#S = N\times (\#Z).$ Set $ Z=\{a,b\}$ and
$S=\{1, \dots, N, 1', \dots, N'\},$ as two alphabet sets.
 
Suppose that $f: D\to D$  has $N$ local dynamics $f_i:Q_i\subset D\to f_i(Q_i)$ that satisfy the following conditions.

\textbf{Assumption 1:} There exists disjoint horizontal strips $H_a$ and $H_b$ that $f$ images $N$ disjoint vertical strips $V^i\subset Q_i,$ homeomorphically to the horizontal strip $H_a$ and $N$ disjoint 
vertical strips $V^{i'}\subset Q_i,$ homeomorphically to $H_b$ (i.e. $f(V^{i})=H_a,\,f(V^{i'})=H_b ,\, i, i'\in S $). Moreover, the horizontal and vertical boundaries 
are preserved i.e., horizontal and vertical boundaries map to the horizontal and the vertical boundaries respectively.

\textbf{Assumption 2:} For any vertical strip $V$ contained in some  $V^{l}\subset \bigcup _{i\in S}V^i$, the region
$\tilde{V}^{l}:=f^{-1}(V)\cap V^l$ is a vertical strip. Moreover, $d(\tilde{V}^{l})<\alpha_V\, d(V)$ for some $0<\alpha_V<1$.
Similarly, for $H$, any horizontal strip contained in an $H_k\subset \bigcup_{j\in Z}H_j$, the region $\tilde{H}_k:=f(H)\cap H_k,$
is a horizontal strip that
$d(\tilde{H}_k)\leq \alpha_H\,d(H)$ for some $0<\alpha_H<1$.

Our objective is to construct an invariant set contained in $(\bigcup_{i\in S}V^{i}\,) \,\cap\,(\bigcup_{j\in Z}H_{j}).$ First we construct a horizontal backward invariant set $\Lambda_{-\infty}$ and then the vertical forward invariant set $\Lambda^{+\infty}$. The final invariant set is made of the intersection of these two sets.

\textbf{Horizontal invariant set $\Lambda_{-\infty}$:}
Let $\bigcup_{s_{-1}\in Z}H_{s_{-1}}$ be the union of horizontal strips satisfying the Assumptions 1 and 2. For $k\geq 1$,
\begin{itemize}
	\item Let $H_{s_{-k-1}\dots s_{-1}}:=\{p\in Q|\,p\in f^{i-1}(H_{s_{-i}}), \,i=1,2,\cdots, k+1\}.$ 
	%with $V_{s_{-k-1}\cdots s_{-1}}\subset V_{s_{-k} \cdots s_{-1}}\subset \cdots \subset V_{s_{-1}}.$ 
	\item The set $\Lambda_{-k}:= \bigcup_{s_i\in S}\,H_{s_{-k-1}\cdots s_{-1}}$ contains $2^{k}$ horizontal strips where $2^{k-1}$ of them are contained in each $H_j,\, j\in Z.$ 
	\item Assumption 2 asserts that $d(H_{s_{-k}\dots s_{-1}})\leq \alpha_{H}^{k-1}$.
\end{itemize}
When $k\to \infty,$  define $\Lambda_{-\infty}:= \bigcup_{s_{-j}\in Z}\, \,H_{\dots s_{-k-1}\dots s_{-1}}$, where
$$H_{\dots\,s_{-k-1}\,\dots \,s_{-1}}:=\{p\in Q|\, p\in f^{i-1}(H_{s_{-i}}), \,i=1,2,\dots\}.$$
Observe that from Lemma \ref{v-h curve}, 
%the set $\Lambda_{-\infty}$ consists of an infinitely many vertical curves, where 
$d(H_{s_{-k-1} \dots s_{-1}})\to 0,$ as $k\to \infty$.\\
\textbf{Vertical invariant set $\Lambda^{+\infty}$:}
Let $\bigcup_{s_0\in S}V^{s_0}$  be the union of vertical strips satisfying the Assumptions 1 and 2. For $k\geq 1$,
\begin{itemize}
	\item Let $V^{s_0s_1\dots\,s_k}:=\{p\in Q|\, f^{i}(p)\in V^{s_i},\,\,i=0,1,\dots,k\}.$
	\item Let $\Lambda^{k}:= \bigcup_{s_i\in S}V^{s_0s_1\cdots\,s_k}$. Then  for $l_0=\#S=2N,$ the set $\Lambda^{k}$ consists 
of $l_k=(2N){l_{k-1}}$  vertical strips, where exactly $l_{k-1}$ strips are contained in each $V^i,\, i\in S.$ 
	\item Assumption 2 asserts that
	\begin{equation}\label{h-inf}
	d(V^{s_0s_1\cdots\,s_k})\leq \alpha_{V}^{k}.
	\end{equation} 
	%$$d(V^{s_0s_1})\leq \nu_{h}d(V^{s_0})\leq \alpha_{H},\,\,\,\,\,s_0,s_1\in S.$$
\end{itemize}	
When $k\to \infty,$  define $\Lambda^{+\infty}:= \bigcup_{s_i\in S, i=0,1,\cdots}V^{s_0s_1\cdots\,s_k\cdots}$,
where, $V^{s_0s_1\cdots\,s_k\cdots}=\{p\in Q|f^{i}(p)\in V^{s_i},\,i=0,1,\cdots\}.$
By Eq.(\ref{h-inf}), $d(V^{s_0s_1\dots\,s_k})\to 0$ when $k\to \infty$ and
by Lemma \ref{v-h curve}, $V^{s_0s_1\dots\,s_k,\dots}$ is a vertical curve. The invariant set $\Lambda$, over all iterations of $f$ in $Q,$
%i.e. a set of points that remains in $Q$ 
is given by $$\Lambda=\{\Lambda_{-\infty}\cap \Lambda^{+\infty}\}\subset \big\{\big(\bigcup_{i\in S}V^{i}\big)\cap \big(\bigcup_{j\in Z}H_{j}\big)\big\}\subset Q. $$
The set $\Lambda$ is uncountable and in fact is a Cantor set. In  Subsection \ref{subsec:3.1}, we perform a conjugacy map $\phi$ which clarifies this fact. 

\subsection{The conjugacy map $\phi$}\label{subsec:3.1}
In this subsection we construct a conjugacy map $\phi$ between the horseshoe map $f:\Lambda \to \Lambda$ and the zip shift map $\sigma_{\tau}:\Sigma \to \Sigma$. 
%Recall the horseshoe set $\Lambda$ from Section \ref{sec:3}. 
For $p\in\Lambda$, let
%there exist two (and only two) infinite sequences,  such that 
$$p=H_{\cdots s_{-k}s_{-2}\cdots s_{-1}} \cap V^{s_0s_1\cdots s_k\cdots},$$
where $s_{-i}\in Z,\,i=1,2,\cdots$ and $s_{i}\in S,\,i=0,1,2,\cdots$. Indeed, we associate to any point $p\in \Lambda$  a bi-infinite sequence, 
over $Z\cup S$. In other words, due to Lemma \ref{v-h intersection}, there exists a well-defined map $\phi$, 
\vspace{0.1cm}
\begin{center}
	\begin{tabular}{rl}
		$\phi:$ & $\Lambda \longrightarrow \Sigma$ \\
		& $p\longmapsto (\dots s_{-k}\dots s_{-1}\,\textbf{.}\,s_0\,s_1\dots s_k\dots).$
	\end{tabular}
\end{center}
\begin{proposition}
	The map $\phi:\Lambda \to \Sigma$ is a homeomorphism.
\end{proposition}	
\begin{proof}
	As $\Lambda$ is a closed subset of $Q,$ it is sufficient to show that $\phi$ is $1-1$, onto, and continuous.
	
	\textbf{$\phi$ is 1-1.} Let $p,p'\in \Lambda$. If $p\neq p'$, then $\phi(p)\neq \phi(p').$ By contradiction, assume that $p\neq p'$ 
and $\phi(p)=\phi(p')=(\dots s_{-n}\dots s_{-1}.\,s_0\,s_1\dots s_n\dots).$
	The construction of $\Lambda$ and Lemma \ref{v-h intersection} imply that points $p,p'$ represent the unique intersection of a vertical curve with a horizontal curve, 
which means $p=p'$ and contradicts the assumption. 
	
	\textbf{$\phi$ is Onto.} 
	For any $s=(\dots s_{-k}\dots s_{-1}\,.\,s_0s_1\dots s_k\dots)\in \Sigma,$ consider the vertical curve $V^{s_0s_1\dots s_k\dots}\in \Lambda^{+\infty}$
 and the horizontal curve $H_{\dots s_{-k}\dots s_{-2} s_{-1}}\in\Lambda_{-\infty}$. 
 By Lemma \ref{v-h intersection}, $V^{s_0s_1\dots s_k\dots}\cap H_{\dots s_{-k}\dots s_{-2} s_{-1}}$ is a unique point $p\in \Lambda$. 
 Thereupon,  $$\phi(p)=(\dots s_{-n}\dots s_{-1}.\,s_0\,s_1\dots s_n\dots).$$ 
	
	\textbf{$\phi$ is continuous.} Given any point $p\in \Lambda,$ and $\epsilon>0,$ one shows that there exists a $\delta>0$ such that $|p-p'|<\delta$
 implies $d(\phi(p), \phi(p'))<\epsilon,$ where $|.|$ is the usual distance in $\mathbb{R}^2$ and $d(.,.)$ is the metric on $\Sigma$.
	Let $\epsilon>0$. By Lemma \ref{dis-equal}, for $|\phi(p)-\phi(p')|<\epsilon$ to occur, there should exist an integer $M$ such that 
if $\phi(p)=(\dots s_{-n}\dots s_{-1}\,.\,s_0\dots s_n\dots )$ and $\phi(p')=(\dots t_{-n}\dots t_{-1}\,.\,t_0\dots t_n\dots )$, then $s_i=t_i$ for $|i|\leq M.$ 
However, the construction of $\Lambda$ implies that the points $p$ and $p'$ lie in the set defined by strips $V^{s_0\dots s_M}\cap H_{s_{-M}\dots s_{-1}}.$ 
	Denote the boundary curves of the strip $H_{s_{-1}\dots s_{-M}},$ by graphs $y=h_{1}(y), y=h_{1}^{'}(x)$ and the boundary curves of the 
strip $V^{s_0\dots s_M},$ by graphs $x=v_1(y),\,x=v_1^{'}(y).$  By Assumption 2,
	\begin{align}\label{v-h max}
	d(V^{s_0\dots s_M})\leq \alpha_{V}^{M} \,\,\,\,\,\,\,\,\text{and}\,\,\,\,\,\,\,\,d(H_{s_{-M}\dots s_{-1}})\leq \alpha_{H}^{M-1}.
	\end{align}
	By Eq.(\ref{dv}), Eq.(\ref{strip norms}) and Eq.(\ref{eq:4.8}),
	\begin{align*}
	\Vert h_1-h_1'\Vert\leq \alpha_{H}^{M-1}\,\,\,\,\,\,\,\text{and}\,\,\,\,\,\,\,\,\Vert v_1-v_1'\Vert\leq\alpha_{V}^{M}.
	\end{align*}
	Next, by Lemma \ref{4.8} and Eq.(\ref{v-h max}) the continuity of $\phi$ follows. Let $z_1$ denote the intersection of $h_1(y)$ with $v_1(x)$ and $z_2$ denote the intersection of
 $h_1^{'}(y)$ with $v_1^{'}(x)$. Since $p$ and $p'$ lie in the intersection of the horizontal and vertical strips  $ H_{s_{-M}\cdots s_{-1}}$ and $V^{s_0\cdots s_M}$, it follows that,	
	$$|p-p'|\leq |z_1-z_2|\leq \frac{1}{1-\mu}(\Vert v_1-v_1'\Vert+\Vert h_1-h_1'\Vert)\leq \frac{1}{1-\mu}(\alpha_{H}^{M-1}+\alpha_{V}^{M}).$$	
	Hence $\phi$ is a homeomorphism.
\end{proof}

The following theorem gives the sufficient condition for the existence of an invariant horseshoe set. As mentioned before, instead of the two horizontal strips (i.e. $Z=\{a,b\}$), 
one can consider any $k$ vertical strips (i.e. $Z=\{a_1,\cdots,a_k\}$) and find  an invariant Cantor set $\Lambda$ for $f$ on which it is topologically conjugate to a  zip shift. 
\begin{theorem}\label{Thm:3}
	Suppose that $f$ is an $N$-to-1 local homeomorphism, which satisfies the Assumptions 1 and 2. Then, $f$ has an invariant Cantor set $\Lambda$, which is 
topologically conjugate to a zip shift map, i.e. the following diagram commutes. 
	\begin{equation}\label{1-0}
	\begin{tikzcd}
		[row sep=tiny,column sep=tiny]& & & \\ \Lambda \arrow{rr}{f}\arrow{dd}{\phi} & & \Lambda\arrow{dd}{\phi} \\& \circlearrowright &  & \\ \Sigma\arrow{rr}{\sigma_{\tau}} & & \Sigma\\
	\end{tikzcd}
	\end{equation}
Here, $\sigma_{\tau}$ is a zip shift map and $\phi$ is a homeomorphism mapping $\Lambda$ onto a zip shift space $\Sigma$.
\end{theorem}
\begin{proof}
	The construction of the homeomorphism $\phi$ is shown in Section \ref{subsec:3.1}. It remains to show that Diagram (\ref{1-0}) commutes.
	%	\textit{ The Diagram \ref{1-0} commutes.} 
	For any $p\in \Lambda$,  
	\begin{equation}\label{phi}
	\phi(p)=(\cdots s_{-k}\cdots s_{-1}.\,s_0\,s_1\cdots s_k\cdots ).
	\end{equation}
	Applying the zip shift map $\sigma_{\tau}$ defined in \ref{sigmatau}, we obtain,
	\begin{equation}\label{tauphi}
	\sigma_{\tau}(\phi(p))=(\cdots s_{-k}\cdots s_{-1}\,\tau(s_0)\,.\,s_1\cdots s_k\cdots).
	\end{equation}
	Observe that for $p=H_{\cdots s_{-k}\cdots s_{-1}} \cap V^{s_0 s_1\cdots s_k\cdots},$ $$f(p)=f(H_{\cdots s_{-k}\cdots s_{-2} s_{-1}} \cap V^{s_0s_1\cdots s_k\cdots})=H_{\cdots s_{-k}\cdots s_{-1}\tau(s_{0})}\cap V^{s_1\cdots s_k\cdots}.$$ 
	Using the definition of $\phi$,
	$$\phi(f(p))=(\cdots s_{-k}\cdots s_{-1}\tau(s_{0})\,.\,s_1\cdots s_k\cdots),$$
	which completes the proof of the Theorem.
\end{proof}

\subsection{Sector Bundles}
In this section we modify Assumption 2 to Assumption 3, which is based only on the properties of the derivative of $f_k$ and is useful when $f$ has a differentiable structure.

Set $V^1,\cdots, V^N,V^{1'}\cdots, V^{N'}, H_a, H_b$ represent an $N$-to-1, HV-strip with Lipschitz constants $\mu_v,\mu_h$. 
Let $f$ map $V^1,\cdots, V^N$ to $H_a$ and map $V^{1'}\cdots, V^{N'}$ to $H_b$ diffeomorphically. 
Recall that $V^{1'}\cdots, V^{N'}$ are copies of $V^1,\cdots, V^N$. 
%{\color{red} Without loss of generality, let us study $V^1,\cdots, V^N$ jointly with $ H_a$ and $H_b$}.
 Define $f(H_l)\cap H_m=H_{lm}$ and $f(V^j)\cap V^i=V_{ij}\in f^{-1}(H_{lm}),$ for $i,j \in S$ and $l,m\in Z$ (see Figure \ref{fig:5}).
Moreover, let $$\mathcal{V}=\bigcup_{i,j\in S}V_{ij},\,\,\,\,\,\,\,\,\mathcal{H}=\bigcup_{l,k\in Z}H_{lm}.$$
Then $f(\mathcal{H})=\mathcal{V}$ diffeomorphically.
For some $z_0\in \mathcal{H}\cap\mathcal{V}$ the unstable and stable cones will be  defined  as follows.
\begin{equation}\label{unstble sector}
	\mathcal{S}^{u}_{z_0}=\{(x_{z_0},y_{z_0})\in \mathbb{R}^2||x_{z_0}| < \mu_v\, |y_{z_0}|\}.
\end{equation}
\begin{equation}\label{unstbl sector}
	\mathcal{S}^{s}_{z_0}=\{(x_{z_0},y_{z_0})\in \mathbb{R}^2||y_{z_0}|< \mu_h\, |x_{z_0}|\}.
\end{equation}

Consider the union of the unstable and stable cones over the points of $\mathcal{H}$ and $\mathcal{V}$ as follows.
$$\mathcal{S}^{u}_{\mathcal{H}}=\bigcup_{z_0\in \mathcal{H}}\mathcal{S}^{u}_{z_0},\,\,\,\mathcal{S}^{s}_{\mathcal{H}}=\bigcup_{z_0\in \mathcal{H}}\mathcal{S}^{s}_{z_0},$$
$$\mathcal{S}^{u}_{\mathcal{V}}=
\bigcup_{z_0\in\mathcal{V}}\mathcal{S}^{u}_{z_0},\,\,\,\mathcal{S}^{s}_{\mathcal{V}}=\bigcup_{z_0\in \mathcal{V}}\mathcal{S}^{s}_{z_0}.$$

\begin{figure}[h]
	\centering
	\includegraphics[scale=0.5]{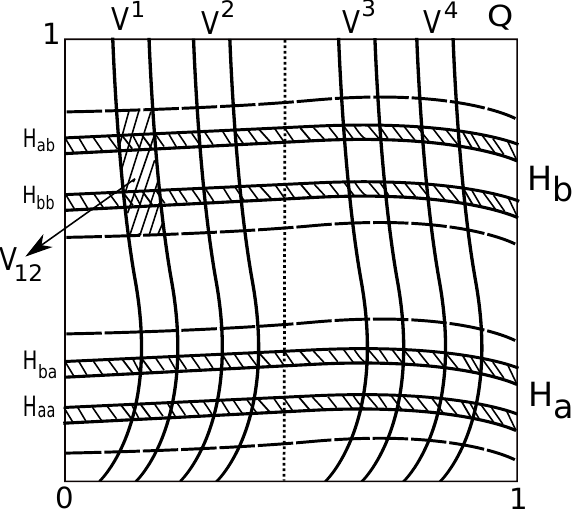}
	%\vspace{-1cm}
	\caption{$H_{lm}$ and $V_{ij}$ for $n = 2$.}
	\label{fig:5}
\end{figure}

The alternative to Assumption 2, is as follows.\\
\textbf{Assumption 3:}  For any $k=1,\dots,m$, $df(\mathcal{S}^{u}_{\mathcal{V}})\subset \mathcal{S}^{u}_{\mathcal{H}}$ and 
$df_k^{-1}(\mathcal{S}^{s}_{\mathcal{H}})\subset \mathcal{S}^{s}_{\mathcal{V}}$.

Moreover, if $(x_{z_0},y_{z_0})\in \mathcal{S}^{u}_{z_0}$ and $d_{z_0}f(x_{z_0},y_{z_0})= (x_{f(z_0)},y_{f(z_0)})\in \mathcal{S}^{u}_{f(z_0)},$ 
then $$|y_{f(z_0)}|\geq (1/\mu)|y_{z_0}|.$$
Similarly for $k=1,\dots,n$, 
if,\\ $(x_{z_0},y_{z_0})\in \mathcal{S}^{s}_{z_0}$ and 
$d_{z_0}f_k^{-1}(x_{z_0},y_{z_0})= (x_{f_k^{-1}(z_0)},y_{f_k^{-1}(z_0)})\in \mathcal{S}^{s}_{f_k^{-1}(z_0)},$ then, $$|y_{f_k^{-1}(z_0)}|\geq (1/\mu)|y_{z_0}|,$$
where $0<\mu < (1-\mu_h\,\mu_v).$ 

\begin{theorem}\label{Thm:4}
	Assumption 1 and 3 with $0<\mu < (1-\mu_h\,\mu_v)$ implies  Assumption 2.
\end{theorem}
\begin{proof}
We prove  the theorem for the horizontal strips. The proof of the other case is analogous.
	
	Let $V$ be a $\mu_v$-vertical strip in $\bigcup_{i\in S}V^i$. Set $\tilde{V}^j=f_{k}^{-1}(V)\cap V^i$. 
Observe that when $f$ is an $N$-to-1 local diffeomorphism, any vertical strip $V$ has $N$ pre-images, but using the fact that $f$ is an $N$-to-1 local diffeomorphism, 
it guarantees that  $\tilde{V}^j$ in $V^i$ is unique. The strip $V$ intersects any horizontal strips $H_k, k\in Z.$ 
In special, it intersects their horizontal boundaries and using Assumption 1, one deduces that for any $j\in S,$ the region $f_{k}^{-1}(V)\cap \tilde{V}^j$ is a vertical strip which
 is a $\mu_v$-vertical strip. To see this,  assume that $(y_1,x_1), (y_2,x_2)$ are two arbitrary points that belong to
 a vertical boundary curve of ${V}^j$, for some fixed $j.$ Then Assumption 3 and the Mean Value Theorem gives
	$$|x_2-x_1|\leq \mu_v\,|y_2-y_1|.$$
By Assumption 3, $df_{k}^{-1}\mathcal{S}^{s}_\mathcal{H}\subset\mathcal{S}^{s}_\mathcal{V}.$ 
Indeed,  the boundaries of $\tilde{V}^j$ are the graphs of some functions $x={v}_1(y),\,x={v}_2(y).$
	
	Next, let $p_0$, $p_1$ be two points in boundaries of ${V}^j$, with the same $y$-coordinate. It is obvious that $d({V}^j)= |p_0-p_1|.$ 
Let $p(t)=p_1t+(1-t)p_0, \,0\leq t\leq 1$ be the vertical line which connects $p_0$ and $p_1$. Then $\dot{p}(t)=p_1-p_2$ belongs to $\mathcal{S}^{u}_\mathcal{H}.$ 
By Assumption 3, $$df(\mathcal{S}^{u}_\mathcal{V})\subset\mathcal{S}^{u}_\mathcal{H},$$
	which means that $\dot{q}(t)=df(p(t))\dot{p}(t)\in \mathcal{S}^{u}_\mathcal{H}$ and $q(t)=f(p(t))$ is a $\mu_h$-horizontal curve. Set $f(p(0))=q_0=(x_1,y_1)$ and 
	$ f(p(1))=q_1=(x_2,y_2).$ They belong to the boundary curves of $V$ which are the graphs of functions $x=v_1(y)$ and $x=v_2(y).$ Then  Lemma \eqref{4.8} gives
	$$  |x_1-x_2|\leq \frac{1}{1-\mu_v\mu_{h}}\big(\Vert v_1-v_2\Vert\big)=\frac{1}{1-\mu_v\mu_{h}} d(V).$$
Moreover, from Assumption 3, 
$$|\dot{q}(t)|\geq \frac{1}{\mu}|\dot{p}(t)|=\frac{1}{\mu} |p_1-p_0|.$$
Indeed,
$$ |p_1-p_0|\leq \mu \int_{0}^{1} |\dot{q}(t)| dt \leq |x_1-x_2|.$$	
Therefore, for  $\alpha_V=\frac{\mu}{1-\mu_v\mu_h}$,$$d({V}^j)\leq \alpha_V\,d(V).$$	
\end{proof}
\begin{figure}[h]
	\centering
	\includegraphics[scale=1.1]{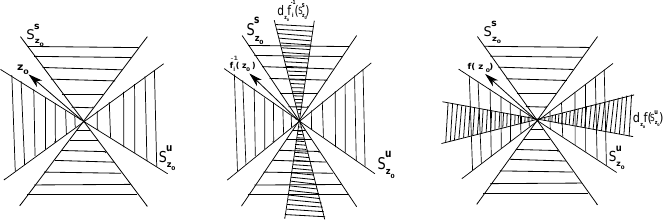}
	%\vspace{-1cm}
	\caption{For $N = 2$, $f_i$ with $i=1,2$.}
	\label{fig:6}
\end{figure}

\section{Structural Stability of the $N$-to-1 horseshoe map}\label{sec:5}
As it is known the co-existence of dense hyperbolic periodic points with strong transversality of stable-unstable manifolds, 
for invertible hyperbolic dynamics is equivalent to the structural stability in the $C^1$ Whitney topology \cite{Robinson:1976}. 
It is believed that hyperbolic endomorphisms in this topology are not structurally stable (see \cite{Przytycki:1976}, \cite{Mane and Pugh:1975}). 

In \cite{Quandt:1989} and in \cite{Berger and Kocsard:2016} the authors give some conditions in which it implies the inverse limit structural stability. 

Let $X$ be a compact Riemannian manifold.  
%with $X_i\subset X, i=1,\dots,m$ open bounded connected subsets with closure, $K_i\subset X_i$ closed subsets. Let also$f_i:X_i\setminus K_i\to f_i(X_i\setminus K_i)$ be $C^r$ diffeomorphisms. Note that whenever $f_i(X_i\setminus K_i)=f_j(X_j\setminus K_j)$ for all $i,j\in\{1,\dots,m\}$, the map is called an m-to-1 $C^r$ local diffeomorphism. 
The inverse limit space is defined as follows, $$X^f=\{\tilde{x}=(x_n)|f(x_{n})=x_{n+1},\,n\in\mathbb{Z}\}.$$
 One defines $\tilde{f}:X^f\to X^f$ being the shift homeomrphism which is called the natural extension map associated with $f$.
   There exists a natural projection $\pi: X^f\to X$ such that for any $\tilde{x}=(x_n)\in X^f$, $\pi(\tilde{x})=x_0=x$. 
   The forward orbit of a point over $f,$ the set $O^+(x)=\{f^n(x)|n\in \mathbb{Z}^+\}$ is unique, but with infinitely many, different pre-histories, 
   in which creates infinitely many points in $\pi^{-1}(x)\subset X^f.$ Let $d(.,.)$ represent the Riemannian metric
    on $X$. For $\tilde{x}=(x_n), \tilde{y}=(y_n)\in X^f$ the following $\tilde{d}$ defines a natural metric on $X^f$.

$$\tilde{d}(\tilde{x},\tilde{y}):=\sum_{n=-\infty}^{\infty}2^{-|n|}d(x_{n}, y_{n}).$$

Let $f$ be a $C^{2}$-and $\Lambda$ be an $f$-invariant closed subset of $X$. One defines 
$$\Lambda^{f}:=\{\tilde{x}=(x_{n})\in X^{f} :  x_{n}\in\Lambda, \,\,for\, all\,\, n\in \mathbb{Z}\}.$$

%\begin{definition}\label{1-4}
Then $\Lambda$ is called a \textit{Uniformly Hyperbolic Set}, if $\forall\tilde{x}\in\Lambda^{f},$ there exist real constants $C>0$, $0<\mu<1$ and for every integer $n, m\in \mathbb{Z}^{+}$ one has:
\begin{itemize}
	\item $T_{x_{n}}X= E^{s}(\tilde{x},n)\oplus E^{u}(\tilde{x},n)$,
	\item $Df(E^{s}(\tilde{x},n))= E^{s}(\tilde{f}(\tilde{x}),n)=E^{s}(\tilde{x},n+1),$\\
	$\Vert Df^{m}_{x_{n}}(v)\Vert\leq C\mu^{m}\Vert v\Vert,$ for $\,v\in E^{s}_{x_{n}}$,
	\item $Df(E^{u}(\tilde{x},n))= E^{u}(\tilde{f}(\tilde{x}),n)=E^{u}(\tilde(x),n+1),$\\
	$ \Vert Df^{m}_{x_{n}}(v)\Vert\geq [C\mu^{m}]^{-1}\Vert v\Vert,$ for $\,v\in E^{u}_{x_{n}}.$
\end{itemize}
%The open balls defined by the use of $\tilde{d}$ and the basic open sets of the product topology generate the same topology on $X^f$. Note that the collection of all sets of the form $\pi^{-1}(U)$ generate a sub-basis for the product topology. Indeed the map $\pi: X^f\to X$ is an open surjective map.
We say that hyperbolic set $\Lambda$ is \textit{locally maximal} if it is a finite union of disjoint closed invariant subsets, each of which being transitive, i.e. having a dense orbit. A hyperbolic set $\Lambda$ is a \textit{basic set} if it is locally maximal. The following Definitions are from \cite{Berger:2018}.

\begin{definition}
An endomorphism satisfies axiom A if its non-wandering set is a finite union of basic sets.
\end{definition}

\begin{definition}
	A $C^r$ map $f:X\to X$ is structurally stable if every $C^r$-perturbation $f':X\to X $
	of the dynamics is conjugate to $f$, i.e. there exists a homeomorphism $h$ so that $h\circ f = f\circ h$.
\end{definition}

\begin{definition}
	The endomorphism $f:X\to X$ is said to be $C^r$-inverse limit stable if for every $C^r$-perturbation $f'$
	of $f$, there exists a homeomorphism $h$ from $X^f$ onto $X^f$ such that $\tilde{f'}\circ h=h\circ \tilde{f}$.
\end{definition}

Let $\Lambda$ be a hyperbolic invariant set. Then one can define the unstable manifold of every point $\tilde{x}=(x_n)\in\Lambda^f$ with
$$W^{u}(f,\tilde{x})=\{\tilde{y}\in X^f: \tilde{d}(x_n,y_n)\xrightarrow[]{n\to-\infty}0\}.$$
When $f$ satisfies axiom A, it is an actual sub-manifold embedded in $X^f$. Moreover, the projection $\pi(W^u (\tilde{x}; f))$ displays a differentiable structure.
 The stable manifolds do not depend on the branches  and are defined similar to the case of $C^r$-diffeomorphisms. In special, for any $x\in \Lambda$ and $\tilde{x}\in\pi^{-1}(x)$ and some small enough $\epsilon>0$, we have 
$$\pi(W^{s}_{\epsilon}(f,\tilde{x}))=W^{s}_{\epsilon}(f,x)=\{y\in X: d(f^n(x),f^n(y))<\epsilon,\,with\,n\to\infty\}.$$

\begin{definition}
An axiom A endomorphism $f$ satisfies the weak transversality condition if for
every $\tilde{x}\in\Omega^f$ ($\Omega^f$ represents the inverse limit set of the $\Omega$-limit set) and any $y\in\Omega_f$  the map $\pi(W^u (\tilde{x}; f))$ is transverse to $W^{s}_{\epsilon}(f,y)$.
\end{definition}

The following theorem is from \cite{Berger and Kocsard:2016}.

\begin{theorem}\label{Thm:5.1}
	If a $C^1$-endomorphisms of a compact manifold satisfies
	axiom A and the weak transversality condition, then it is inverse limit stable.
\end{theorem}

In what follows we show that the $N$-to-1 horseshoe map $f_{|_{\Lambda}}$ is inverse limit structurally stable.

\begin{theorem}
Let $f_{|_{\Lambda}}:\Lambda\to\Lambda$ be an $N$-to-1 horseshoe map. Then $f_{|_{\Lambda}}$ is inverse limit stable.
\end{theorem}

\begin{proof}
Note that the $N$-to-1 horseshoe set constructed in Section \ref{Sec:2} is a closed invariant subset of  $Q\subset D$ and $f_{|_{\Lambda}}$  is a transitive map by Theorems \ref{thm:2} and \ref{Thm:3}. 
Therefore, $\Lambda$ is a locally maximal basic set. 
As topological conjugacy preserves the structure of the orbits, it is easy to verify that $\Lambda$ is a non-wandering set. 
Indeed, $f_{|_{\Lambda}}:\Lambda\to\Lambda$ is an example of an Axiom A endomorphism. 
%We show that it satisfies the transversality condition. 
By construction, an $N$-to-1 Smale horseshoe set not only is hyperbolic, but also has (weak) strong transversality condition. 
By that, we mean that for any $\tilde{x},\tilde{y}\in \Lambda^f$ and all $n\geq 0$ we have $f^n_{|_{\pi(W^{u}(f,\tilde{x}))}}\pitchfork \pi(W^s_{\epsilon}(f,\tilde{y}))\neq \emptyset$. So by Theorem \ref{Thm:5.1}, the map  $f_{|_{\Lambda}}$ is inverse limit structurally stable.
%for any $x,y\in \Lambda$ and in special for any $\tilde{x}\in 
\end{proof}

The construction of the $N$-to-1 horseshoe map which is presented in Sections \ref{sec:3} and \ref{sec:4}  performs both conditions of strong transversality and density of hyperbolic periodic points. 
It provides not only the inverse limit structural stability, but also seems promising to an $N$-to-1 $C^1$-perturbation. Let  $LD^1(X)$ be the space of all $C^1$ local diffeomorphisms defined on $D$, and $U_{\epsilon}(f)$ represent an $\epsilon$-neighborhood of $f$ in $C^1$-Withney topology. We say that the $N$-to-1 map $f$ has $N$-to-1 $C^1$-structural stability if there exists some $\epsilon>0$ such that for all $g\in U_{\epsilon}^N(f)$, (i.e. the restriction of $U_{\epsilon}(f)$ to all $N$-to-1 maps) there exists a homeomorphism $h$ such that $h\circ f = f\circ h$. 
%In what follows we show that an $N$-to-1 $C^1$-perturbation of an $N$-to-1 horseshoe map satisfies Assumptions 1 and 3 (See Figure \ref{fig:1} and \ref{fig:2} for a 2-to-1 horseshoe map).

%\begin{lemma}
%	Let $f:D\to D$ be a diffeomorphism which satisfies the Assumptions 1. Set $g\in U_{\epsilon}^N(f)$ as an N-to-1 $C^1$-perturbation of $f$. Then $g$ satisfies the Assumption 1.
%\end{lemma}

\begin{theorem}
	Let $f:D\to D$ be an $N$-to-1 local diffeomorphism that satisfies the Assumptions 1 and 3. Then $f_{|_{\Lambda}}$ is $N$-to-1 $C^1$-structurally stable in the $C^1$-Whitney topology.
\end{theorem}

\begin{proof}
	
Let $V^1,\dots,V^N$ and $V^{1'},\dots,V^{N'}$ represent the vertical strips of  map $f:D\to D$ that satisfies Assumptions 1 and 3. Set $g(X)=f(x,y)+\epsilon(x,y)$  as an small enough $N$-to-1 perturbation of $f$ such that $g(Q)\cap Q\subset D$ is a horseshoe with two disjoint horizontal strips (for a 2-to-1 perturbation see figure \ref{fig:7}). Without any loss of generality we denote these new horizontal strips by $H_a$ and $H_b$. 
%It is an $N$-layer horseshoe, thus it has $N$ disjoint vertical horseshoes which are its pre-images. 
Note that by definition, an $N$-to-1 perturbation of $f$ is an $N$-to-1 local diffeomorphism $C^1$-close to it such that if $f_1,\dots,f_N$ are the local dynamics (see Definition \ref{Def:p-t-1}) of $f$ then there exists $N$ diffeomorphisms $g_1,\dots, g_N$ which are local dynamics associated with $g$.
 This implies that the pre-images of $H_a$ and $H_b$ under $g$ contains $2N$ vertical strips satisfying  Assumption 1 (i.e.
the vertical strips $V^1,\dots,V^N$ and $V^{1'},\dots,V^{N'}$ are homeomorphically mapped to new horizontal strips $H_a$ and $H_b$ and the horizontal and vertical boundaries are preserved).

%Let $g(x)=f(X)+\epsilon \cos(2\pi)\,X$ be a small perturbation of $f$. 
An $N$-to-1 $C^1$ perturbation of $f$ and consequently the $C^1$ perturbation of diffeomorphisms $f_i,i=1,\dots ,N$ preserves the cone properties. 
Thus, it is enough to choose $0<\epsilon<e$ (where $e>0$ is the expansivity constant of $f$, see Proposition
 \ref{prop:2})  small enough such that for $\mu_g=\mu_{h}^g\mu_{v}^g$ the perturbed map $g,$ satisfies  Assumption 3.
  Moreover, using Theorem \ref{Thm:4}, 
%and $\alpha_V^g=\frac{\mu_g}{1-\mu_v^g\mu_h^g}$ 
 Assumption 2 is satisfied as well. Thereupon, by Theorem \ref{Thm:3}
  the restriction of $g$ to the invariant set $\Lambda_g$ is topologically
   conjugate to a zip shift map on two sets of alphabets $S_g$ and $Z_g$ with
    $\#(S_g)=2N$ and $\# (Z_g) = 2$.
     Since $N$-to-1 full zip shift maps with the same
      cardinalities for $S$ and $Z,$ are topologically conjugate, so, $f_{|_{\Lambda}}$ is  $C^1$ structurally stable. 
%As $f_|{\Lambda}$ is topologically conjugated with a full N-to-1 zip shift map, let $|\epsilon(x,y)|< min\{\delta, c\}$ where $\delta$ is the expansivity constant of $f$ (see Propositions \ref{prop:1} and \ref{prop:2}) and $c$ is calculated as follows. Let $d(H)=\min \{d(H_a(f)),d(H_b(f))\}$ then $c<\frac{1}{N^2}d(H)$.

%{\color{red}Let $f_i$ for $i=1,2,\dots, N$ represent the maps that image the $N$ vertical horseshoes diffeomorphically to the horizontal horseshoe. 
%
%%Note that a perturbation $C^1$ of an $N$-to-1 local diffeomorphism defined in a square $Q\subset \mathbb{R}^2$ is not necessarily $N$-to-1.
%In fact, the perturbation can result in a finite-to-1 map. But there is a small enough $\epsilon>0$ such that for any $g\in U_{\epsilon}(f)$, the
% Conley-Moser conditions (Theorems \ref{Thm:4} and \ref{Thm:3}) are satisfied. This means that for any $g\in U_{\epsilon}(f)$, it is possible to choose two horizontal strips, such that the restriction of $g$ to these horizontal strips is $N$-to-1 and their preimages produce $2N$ vertical strips, so that Assumptions 1 and 3 are satisfied.
%Therefore, 
%this implies the structural stability of the $N$-to-1 Smale horseshoe in the $C^1$ Whitney topology. 
%}
\end{proof} 

\begin{figure}[h]
	\centering
	\includegraphics[scale=1]{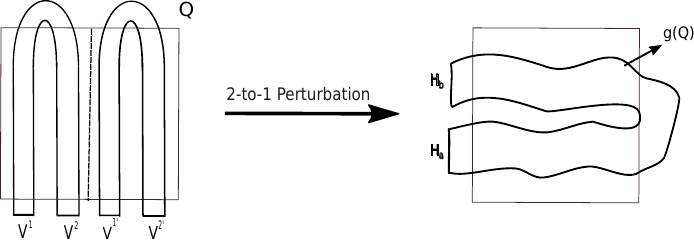}
	%\vspace{-1cm}
	\caption{2-to-1 perturbation of $f$ for $N= 2$.}
	\label{fig:7}
\end{figure}

The $N$-to-1 horseshoe map is one of the first examples of a structurally stable Weak Axiom A map in the sense of \cite{Przytycki:1976}
 for which the restriction of the map to the attractor set is not injective. 
 In \cite{Mehdipour:2024} the natural zip shift map which is a modified natural extension of zip shift  is introduced. 
 In contrary with natural extension of an endomorphism which is a homeomorphism semi-conjugate to the original map, 
 the natural zip shift map is a local homeomorphism which provides the topological conjugacy with the original endomorphism. 
 We aim to use this  to re-study the structural stability of hyperbolic endomorphisms. It is worth mentioning that there are recent works 
 \cite{Kurenkov:2017}, \cite{Grines and Zhuzhoma and Kurenkov:2018} and \cite{Grines and Zhuzhoma and Kurenkov:2021}, 
 where the authors study the Derived from Anosov endomorphisms of the two-dimensional torus. Their achievements are consistent with our results.
\bigskip
\\
\textbf{Conflict of Interest}\\
	The authors declare no conflict of interest.
%\end{conflictsofinterest}

\end{document}